\documentclass[a4paper]{amsart}

\usepackage[utf8]{inputenc}
\usepackage[T1]{fontenc}
\usepackage{kpfonts}
\usepackage{enumerate}
\usepackage[english]{babel}
\usepackage[size=2em,nohug,midshaft,scriptlabels,PostScript=dvips]{diagrams}

\theoremstyle{plain}
\newtheorem{theorem}{Theorem}[section]
\newtheorem{lemma}[theorem]{Lemma}
\newtheorem{proposition}[theorem]{Proposition}
\newtheorem{corollary}[theorem]{Corollary}
\theoremstyle{definition}
\newtheorem{definition}[theorem]{Definition}
\newtheorem{remark}[theorem]{Remark}

\newtheorem{condition}[theorem]{Condition}

\usepackage{accents,mathtools}
\usepackage{mathtools}
\DeclareMathSymbol{\widehatsym}{\mathord}{largesymbols}{"62}
\newcommand\lowerhathatsym{%
  \text{\smash{\rlap{\raisebox{-1.1ex}{$\widehatsym$}}\raisebox{-1.4ex}{%
    $\widehatsym$}}}}
\newcommand\hathatnospace[1]{%
  \mathchoice
    {\accentset{\displaystyle\lowerhathatsym}{#1}}
    {\accentset{\textstyle\lowerhathatsym}{#1}}
    {\accentset{\scriptstyle\lowerhathatsym}{#1}}
    {\accentset{\scriptscriptstyle\lowerhathatsym}{#1}}
}
\newcommand\widehathat[1]{%
   \mathclap{\phantom{\hat{#1}}}{\hathatnospace{#1}}}

\newcommand{\ZZ}{\mathbb{Z}} 
\newcommand{\QQ}{\mathbb{Q}} 
\newcommand{\CC}{\mathbb{C}} 
\newcommand{\iso}{\cong}      
\newcommand{\sheaf}[1]{\mathscr{#1}} 
\newcommand{\OO}{{\sheaf{O}}}   
\newcommand{\tsum}{{\textstyle\sum}}  
\newcommand{\res}[2]{\left.#1\right|_{#2}} 
\newcommand{\isoto}{\overset{\smash{\raisebox{-0.6ex}{$\textstyle\sim\;$}}}{\to}} 
\newcommand{\tensor}{\otimes}
\newcommand{\into}{\hookrightarrow} 
\newcommand{\onto}{\twoheadrightarrow} 
\newcommand{\HHom}{\mathcal{H}\!\mathit{om}} 
\newcommand{\EExt}{\mathcal{E}\!\mathit{xt}} 

\DeclareMathOperator{\ch}{ch} 
\DeclareMathOperator{\Pic}{Pic} 

\DeclareMathOperator{\End}{End}
\DeclareMathOperator{\Aut}{Aut}
\DeclareMathOperator{\Ext}{Ext}
\DeclareMathOperator{\Spec}{Spec}

\DeclareMathOperator{\Ker}{Ker}
\DeclareMathOperator{\tr}{tr} 
\DeclareMathOperator{\Hilb}{Hilb}
\DeclareMathOperator{\DT}{DT} 

\begin{document}

\title[Donaldson--Thomas on abelian threefolds]{Donaldson--Thomas invariants for complexes on abelian threefolds}

\author{Martin G. Gulbrandsen}

\address{Stord/Haugesund University College, Norway}
\email{martin.gulbrandsen@hsh.no}

\begin{abstract}
Donaldson--Thomas invariants for moduli spaces $M$
of perfect complexes on an abelian threefold $X$ are usually zero.
A better object is the quotient $K=[M/X\times\widehat{X}]$ of complexes
modulo twist and translation. Roughly speaking, this amounts to fixing not
only the determinant of the complexes in $M$, but also that of their
Fourier--Mukai transform. We modify 
the standard perfect symmetric obstruction theory for perfect complexes to
obtain a virtual fundamental class, giving rise to a DT-type invariant of
the quotient $K$. It is insensitive to deformations
of $X$, and respects derived equivalence. As illustrations we examine
the case of Picard bundles and of Hilbert schemes of points.
\end{abstract}

\maketitle

\section{Introduction}\label{sec:intro}

The aim of this paper is to attach nontrivial Donaldson--Thomas type invariants
to abelian threefolds.

Donaldson--Thomas invariants, as defined by Thomas \cite{thomas2000}, are
integers associated to (proper) moduli spaces $M$ of stable coherent sheaves on
a projective threefold $X$ with trivial canonical bundle. By definition, the
Donaldson--Thomas invariant is the degree of the virtual fundamental class (of
dimension zero) associated to a natural perfect obstruction theory on $M$.

If $\Pic^0(X)$ is nontrivial, so that line bundles deform, the perfect obstruction
theory contains a trivial summand, causing the virtual fundamental class and
hence its degree to be zero. The standard remedy is to shrink $M$ so that it parametrizes
only sheaves $\sheaf{E}$ with fixed determinant line bundle $\det(\sheaf{E})$.
The resulting invariant makes sense when $X$ is an abelian threefold,
but is still almost always zero, as there is another trivial summand present
in the obstruction theory. Roughly speaking, this summand controls
deformations of the determinant line bundle $\det(\widehat{\sheaf{E}})$
of the Fourier--Mukai transform $\widehat{\sheaf{E}}$. In this text
we call $\det(\widehat{\sheaf{E}})$ the \emph{codeterminant} of $\sheaf{E}$.

In Section \ref{sec:det-codet} we modify the standard perfect obstruction theory to
obtain a virtual fundamental class on moduli spaces for sheaves --- and more
generally perfect complexes --- on abelian threfolds, with fixed determinant and codeterminant. This
gives rise to a DT-type invariant which is nontrivial in general. In Section
\ref{sec:K} we package the data in a more natural way by forming the stack
quotient $K = [M/X\times\widehat{X}]$ and assigning a DT-invariant to $K$. This
invariant is insensitive to deformations of $X$, and it agrees with Behrend's
weighted Euler characteristic of $K$. In particular it is intrinsic to $K$.
Moreover, it
respects derived equivalence in the following sense: if $Y$ is a second
abelian threefold such that the derived categories $D(X) \iso D(Y)$ are equivalent,
then $M$ may equally well be considered as a moduli space for complexes on $Y$,
but the quotient $K$ and its DT-invariant stay the same.

As indication that the quotient $K$ is a natural object, and that its
DT-invariant is nontrivial in general, we examine two examples in Section
\ref{sec:examples}: Picard bundles, and Hilbert schemes of points.

\subsection{Notation}
Let $\pi\colon X\to S$ be a projective abelian scheme over a separated, noetherian
and connected base scheme $S$, itself defined over an algebraically closed
field $k$ of characteristic zero.

In some sections we let $k=\CC$ to apply
a result of Mukai (Section \ref{sec:chern}),
relating the Chern character of a complex with that of its Fourier--Mukai transform.
Furthermore, the characteristic zero assumption is used to ensure
that line bundles deform freely.

The group law is written
$m\colon X\times_S X\to X$,
and translation along an $S'$-valued point $x\in X(S')$ is written
$T_x\colon X_{S'}\to X_{S'}$. We let $\sheaf{P}$
denote the normalized Poincar\'e bundle on $X\times_S\widehat{X}$, where
$\widehat{\pi}\colon \widehat{X}\to S$
is the dual abelian scheme.
When $\sheaf{L}$ is a line bundle on $X$, we let $\phi_{\sheaf{L}}\colon X\to\widehat{X}$
be the $S$-morphism defined by the line bundle $m^*\sheaf{L} \tensor p_1^*\sheaf{L}^{-1}\tensor p_2^*\sheaf{L}^{-1}$.
We denote by $\sheaf{P}_\xi$ the restriction $(1\times\xi)^*\sheaf{P}$ for a valued point $\xi\in \widehat{X}(S')$.
Abusing notation slightly,
we also write $\sheaf{P}_x$ for $(x\times 1)^*\sheaf{P}$ when $x\in X(S')$.

We write the Fourier--Mukai transform of a complex $\sheaf{E}$ of $\OO_X$-modules
as $\widehat{\sheaf{E}}$; it is a complex
on $\widehat{X}$. This is a deviation from the literature:
the notation $\widehat{\sheaf{E}}$ is usually reserved for WIT-sheaves.
The relative version of Mukai's Fourier inversion theorem \cite[Theorem 1.1]{mukai87} gives a functorial isomorphism
\begin{equation}\label{eq:fourier}
\widehathat{\sheaf{E}} \iso (-1)^*\sheaf{E}\tensor\omega_{X/S}[-g]
\end{equation}
where $g$ is the relative dimension of $X$ over $S$.

Since the Fourier--Mukai transform may map a sheaf to an honest complex, i.e.~having
nontrivial cohomology in more than one degree,
it is natural to work with complexes from the start.
Complexes will be considered as objects in the derived category, and
functors $f_*$, $f^*$, $\tensor$, $\HHom$ etc.\ will mean their derived versions throughout.
In particular, restriction of a complex to a subscheme, such as a fibre of a morphism,
means derived restriction.
We write $H^i$ and $\Ext^i$ for hypercohomology and hyperext.
The cohomology sheaves of a complex will be denoted $h^p(-)$.
Thus, the classical $p$'th derived functor of, say $f_*$, will be written $h^pf_*(-)$.

We use line bundle and vector bundle as synonyms for invertible sheaf
and locally free sheaf.

\subsection{Moduli spaces}\label{sec:moduli}

A complex $\sheaf{E}$ of sheaves on a scheme $X/S$
is \emph{perfect} if it is Zariski locally isomorphic to a bounded complex of locally free sheaves of finite rank.
It is \emph{simple} if
\begin{equation*}
\Ext^p_{X_s}(\sheaf{E}_s, \sheaf{E}_s) =
\begin{cases}
k & p=0,\\
0 & p<0
\end{cases}
\end{equation*}
for all $k$-valued points $s\in S$.
(To save ink we thus let simplicity also encompass the property of having no negative Exts.)

By a \emph{moduli space} $M/S$ for simple perfect complexes on $X$,
we mean a subfunctor, representable by scheme or algebraic space of finite type$/S$, of the functor sending a scheme $T$ to
\begin{equation*}
\left\{\parbox{20ex}{perfect simple complexes on $X\times_S T$}\right\}
\bigg/\parbox{20ex}{twist by $\Pic(T)$ and quasi-isomorphism}
\end{equation*}
and which is \emph{locally complete}:
whenever a $T$-valued point $\sheaf{E}\in M(T)$
extends to a complex $\overline{\sheaf{E}}$ on $X\times_S \overline{T}$
for an infinitesimal thickening $T\subset \overline{T}$,
then $\overline{\sheaf{E}}$ is a $\overline{T}$-valued point of $M$.

Weaker notions of moduli spaces suffice, and we will allow Simpson moduli
spaces $M$ for (locally free resolutions of)
stable sheaves, with fixed Hilbert polynomial. This is not quite covered by the
definition above: a universal family $\sheaf{E}$ on $M\times X$ may only
exist (\'etale) locally.
Still the sheaf $\HHom(\sheaf{E},\sheaf{E})$ does exist globally, and we will ignore that the family $\sheaf{E}$ itself
may not exist as a sheaf.
We are not concerned with stability conditions for complexes in general,
but take the moduli space $M$ as given. 

Eventually, we will assume that the moduli space $M$ is proper, as we want
to take degrees of zero cycles on it.

\subsection{Obstruction theory}

For the following construction, we refer to Huybrechts--Thomas'
treatment \cite{HT2010} of obstruction theory for perfect complexes,
which we follow closely.

Let $L_{M/S}$ denote the truncation to degrees $\ge -1$ of the
cotangent complex of $M/S$. If $M$ embeds as a closed subscheme of a
smooth scheme $W$ over $S$, with ideal $\sheaf{I}\subset \OO_W$, this is
\begin{equation*}
L_{M/S}\colon \left( \sheaf{I}/\sheaf{I}^2 \to \res{\Omega_{W/S}}{M} \right)
\end{equation*}
with nonzero objects in degree $-1$ and $0$.

Associated to the
universal complex $\sheaf{E}$ on $X\times_S M$ there is an element,
the Atiyah class of $\sheaf{E}$, in
$\Ext^1(\sheaf{E},\sheaf{E}\tensor L_{\left.X\times_S M\right/S})$.
Since $p_2^*L_{M/S}$ is a direct summand of $L_{\left.X\times_S M\right/S}$,
the Atiyah class projects
to an element in
\begin{align*}
\Ext^1(\sheaf{E},\sheaf{E}\tensor p_2^*L_{M/S})
&\iso
H^1(\HHom(\sheaf{E},\sheaf{E})\tensor p_2^*L_{M/S})\\
&\iso
H^1(p_{2*}\HHom(\sheaf{E},\sheaf{E})\tensor L_{M/S}) \\
&\iso
\Ext^1((p_{2*}\HHom(\sheaf{E},\sheaf{E}))^\vee, L_{M/S})
\end{align*}
and thus defines a morphism
\begin{equation}\label{eq:obstr-morphism}
(p_{2*}\HHom(\sheaf{E},\sheaf{E}))^\vee[-1] \to L_{M/S}
\end{equation}
in the derived category.
(By relative duality, 
the source may be identified with
$p_{2*}\HHom(\sheaf{E},\sheaf{E}\tensor p_1^*\omega_{X/S})[g-1]$,
as is commonly done in the literature; we prefer not to.)

Replacing $\HHom$ with the kernel of the trace map,
and $M$ with its subspace $M(\sheaf{L})$ of complexes with fixed determinant,
the morphism \eqref{eq:obstr-morphism} becomes
a relative obstruction theory \cite[Theorem 4.1]{HT2010}
in the sense of Behrend--Fantechi \cite{BF97}.

\section{Deformations with fixed determinant and codeterminant}\label{sec:det-codet}

\subsection{Chern class condition}\label{sec:chern}

In this section, $S=\Spec \CC$ and $X$ is an abelian variety of dimension $g$.
Let $\ch$ be the Chern character in $\bigoplus_p H^{2p}(X,\ZZ)$ of a perfect complex $\sheaf{E}$,
with homogeneous decomposition
\begin{equation*}
\ch = r + c_1 + \cdots + \gamma + \chi.
\end{equation*}
We view the rank $r$ and the Euler characteristic $\chi$ as integers.
By a result of Mukai \cite[Corollary 1.18]{mukai87},
the Chern character of the Fourier--Mukai transform $\widehat{\sheaf{E}}$ is
\begin{equation*}
\widehat{\ch} = \chi - \gamma + \cdots + (-1)^{g-1} c_1 + (-1)^g r
\end{equation*}
when identifying 
\begin{equation}\label{eq:PD}
H^{2p}(X,\ZZ)\iso H^{2(g-p)}(X,\ZZ)^{\vee} \iso H^{2(g-p)}(\widehat{X},\ZZ)
\end{equation}
via Poincar\'e duality.

To the divisor class $c_1$ there is an associated homomorphism $\phi_{c_1}\colon X\to \widehat{X}$.
Via \eqref{eq:PD}, the curve class $\gamma$ corresponds to a divisor class on $\widehat{X}$,
(the first Chern class of $\widehat{\sheaf{E}}$, up to sign);
denote the associated homomorphism by $\psi_\gamma\colon \widehat{X}\to X$.

Now let $M$ be a moduli space of perfect complexes with Chern character $\ch$, big
enough to be stable under the actions of $X$ by translation and
$\widehat{X}=\Pic^0(X)$ by twist: a (valued) point $(x,\xi)\in X\times\widehat{X}$
acts on complexes $\sheaf{E}\in M$ by
\begin{equation*}
\sheaf{E} \mapsto T_{-x}^*(\sheaf{E})\tensor\sheaf{P}_\xi,
\end{equation*}
where the sign on $x$ is for cosmetic reasons.

The following is modelled on the constructions of Beauville \cite[Section 7]{beauville83} and
Yoshioka \cite[Section 4.1]{yoshioka2001}
in the case of sheaves on abelian surfaces.
The determinant $\det(\sheaf{E})$ and codeterminant $\det(\widehat{\sheaf{E}})$
line bundles of the universal family on $M\times X$
define a morphism
\begin{equation}\label{eq:delta}
\delta\colon M \to \Pic^{-\gamma}(\widehat X) \times \Pic^{c_1}(X)
\end{equation}
($\gamma$ is viewed as a divisor class on $\widehat X$).
We want to use the action of $X\times\widehat{X}$ on $M$ to ensure that all fibres of
$\delta$ are isomorphic.

\begin{condition}\label{condition}
Assume that
\begin{equation*}
\begin{pmatrix}
\chi & -\psi_{\gamma} \\
-\phi_{c_1} & r
\end{pmatrix} \in \End(X\times\widehat{X})
\end{equation*}
is an isogeny.
\end{condition}

The condition is rather weak, as we will indicate in Proposition \ref{prop:semihom}.
Recall that an \emph{isotrivial fibration} is a morphism that
becomes the projection from a product after an \'etale base change.

\begin{proposition}\label{prop:isotrivial}
Impose Condition \ref{condition}. Then $\delta$ is an isotrivial fibration.
\end{proposition}

\begin{proof}
Let $\sheaf{E}\in M$ and $(\sheaf{L}',\sheaf{L}) = \delta(\sheaf{E})$. Then
\begin{align*}
\det(T_{-x}^*(\sheaf{E})\tensor\sheaf{P}_\xi) &=
\underbrace{(T_{-x}^*\sheaf{L}\tensor\sheaf{L}^{-1})}_{\phi_{c_1}(-x)}\tensor\sheaf{P}_{r\xi}\tensor\sheaf{L}\\
\det(\widehat{T_{-x}^*(\sheaf{E})\tensor\sheaf{P}_\xi}) &=
\underbrace{(T_\xi^*\sheaf{L}'\tensor(\sheaf{L}')^{-1})}_{\psi_{-\gamma}(\xi)}\tensor\sheaf{P}_{\chi x}\tensor\sheaf{L}'
\end{align*}
(use that the Fourier--Mukai transform exchanges twist and translation, up to sign \cite{mukai87}).
Thus, if $X\times\widehat{X}$ acts on the base of $\delta$ via the matrix in Condition \ref{condition}
followed by addition in $\Pic(\widehat{X})\times\Pic(X)$,
\begin{equation*}
X\times\widehat{X} \xrightarrow{\text{isogeny}} X\times\widehat{X}
= \Pic^0(\widehat{X}) \times \Pic^0(X) \curvearrowright \Pic^{-\gamma}(\widehat X) \times \Pic^{c_1}(X),
\end{equation*}
then $\delta$ is $X\times\widehat{X}$-equivariant.

It follows that, for any fibre $M(\sheaf{L}',\sheaf{L}) = \delta^{-1}(\sheaf{L}',\sheaf{L})$
of $\delta$, 
there is a Cartesian diagram
\begin{equation*}
\begin{diagram}
(X\times\widehat{X}) \times M(\sheaf{L}',\sheaf{L}) & \rTo^{\text{action}} & M \\
\dTo_{\text{projection}} &  & \dTo_\delta \\
X\times\widehat{X} & \rTo^{\text{isogeny}} & \Pic^{-\gamma}(\widehat X) \times \Pic^{c_1}(X)
\end{diagram}
\end{equation*}
where the isogeny at the bottom is defined by the action on $(\sheaf{L}',\sheaf{L})$.
\end{proof}

Recall the endomorphism construction studied by Morikawa \cite{morikawa54} and Matsusaka
\cite{matsusaka59}: if $\sigma$ and $\tau$ are cycles of complementary dimension on
an abelian variety $X$, then $\alpha=\alpha(\sigma,\tau)$ is the endomorphism
\begin{equation*}
\alpha(x) = \tsum \sigma\tau_x - \tsum \sigma\tau,
\end{equation*}
where $\tau_x = T_x(\tau)$ is the translated cycle, and
the sum means addition of zero cycles using the group law on $X$.
This endomorphism
depends only on the numerical equivalence classes of $\sigma$ and $\tau$ \cite[Theorem 1]{matsusaka59}.
Condition \ref{condition} may be checked on the ``determinant'' of the matrix appearing:

\begin{lemma}\label{lemma:isogeny}
Condition \ref{condition} holds if and only if the homomorphism
\begin{equation*}
r\chi - \alpha(\gamma, c_1)\in \End(X)
\end{equation*}
is an isogeny.
\end{lemma}

\begin{proof}
The homomorphism in Condition \ref{condition} is an isogeny if and only if
\begin{equation*}
\begin{pmatrix}
r & \psi_{\gamma} \\
\phi_{c_1} & \chi
\end{pmatrix} \in \End(X\times\widehat{X})
\end{equation*}
is an isogeny, and their composition is
\begin{equation*}
\begin{pmatrix}
\chi & -\psi_\gamma \\
-\phi_{c_1} & r
\end{pmatrix}
\begin{pmatrix}
r & \psi_{\gamma} \\
\phi_{c_1} & \chi
\end{pmatrix}
=
\begin{pmatrix}
\chi r - \psi_\gamma\phi_{c_1} & 0 \\
0 & \chi r - \phi_{c_1}\psi_{\gamma}
\end{pmatrix}.
\end{equation*}
Thus Condition \ref{condition} holds if and only if this diagonal matrix is an isogeny.
The two entries on the diagonal are duals of each other, so to prove the lemma it suffices to see that
\begin{equation}\label{eq:endomorphism}
\psi_\gamma\circ \phi_{c_1} = \alpha(\gamma,c_1).
\end{equation}

Both sides of \eqref{eq:endomorphism} are $\ZZ$-linear in $\gamma$.
Thus it suffices to treat the case where $\gamma$ is the class of an integral curve $C\subset X$.
Let $\nu\colon \widetilde{C}\to X$ be the normalization of $C$, and let $J$ be the Jacobian of $\widetilde{C}$.
There are induced Picard and Albanese maps $\rho\colon \widehat{X}\to J$ and $\sigma\colon J\to X$.

It is straight forward to verify (see proof of \cite[Prop.~11.6.1]{BL2004},
where the Morikawa--Matsusaka endomorphism is defined with opposite sign of
ours) that
\begin{equation*}
\alpha(\gamma, c_1) = -\sigma \circ \rho \circ \phi_{c_1}.
\end{equation*}
Now choose a line bundle $\sheaf{L}$ on $\widetilde{C}$ (taking $\OO_{\widetilde{C}}$ is ok).
By Grothendieck--Riemann--Roch for $\nu$, we have
\begin{equation*}
\big[\ch(\nu_*\sheaf{L})\big]_{g-1} = \gamma
\end{equation*}
where the subscript $g-1$ denotes the homogeneous component in $H^{2(g-1)}(X,\ZZ)$.
Thus the first Chern class of $\widehat{\nu_*\sheaf{L}}$ corresponds to $-\gamma$ via Poincar\'e duality \eqref{eq:PD},
so with $\sheaf{L}' = \det(\widehat{\nu_*\sheaf{L}})$, the homomorphism $\phi_{\sheaf{L}'}\colon \widehat{X}\to X$ coincides with $-\psi_\gamma$.
But now \cite[Prop.~17.3]{polishchuk2003}
\begin{equation*}
\phi_{\sheaf{L}'} = \sigma\circ \rho,
\end{equation*}
and we have established \eqref{eq:endomorphism}.
\end{proof}

\begin{remark}[{Matsusaka \cite[Proposition 1]{matsusaka59}}]\label{rem:alpha}
If $\gamma$ is proportional to $c_1^{g-1}$ in $H^{2(g-1)}(X,\QQ)$, then
$\alpha(\gamma,c_1)$ is multiplication by the integer $\deg(\gamma c_1)/g$.
So in this case the endomorphism in Lemma \ref{lemma:isogeny} is either zero
or an isogeny.
\end{remark}

\subsection{Diagonals and traces}\label{sec:diagtrace}

Return to the relative situation of an abelian scheme $X/S$.

\begin{lemma}\label{lem:fm-hom}
Let $\sheaf{E}$ and $\sheaf{F}$ be perfect complexes on $X$.
\begin{enumerate}[(i)]
\item There is a canonical isomorphism
\begin{equation*}
\pi_*\HHom(\sheaf{E},\sheaf{F}) \iso \widehat{\pi}_*\HHom(\widehat{\sheaf{E}},\widehat{\sheaf{F}})
\end{equation*}
of complexes in the derived category $D(S)$.\label{item:fm-hom}
\item The induced isomorphism
\begin{equation*}
\Ext_X^i(\sheaf{E},\sheaf{F}) \iso \Ext_{\widehat{X}}^i(\widehat{\sheaf{E}},\widehat{\sheaf{F}})
\end{equation*}
sends an element on the left, viewed as a derived category morphism
$\sheaf{E}\to\sheaf{F}[i]$, to its Fourier--Mukai transform
$\widehat{\sheaf{E}}\to\widehat{\sheaf{F}}[i]$.\label{item:fm-ext}
\end{enumerate}
\end{lemma}

\begin{proof}
Consider the fibre diagram
\begin{equation*}
\begin{diagram}
X\times_S \widehat{X} & \rTo^{p_2} & \widehat{X} \\
\dTo^{p_1} & & \dTo^{\widehat{\pi}} \\
X & \rTo^\pi & S
\end{diagram}
\end{equation*}
and form the complex $\HHom(p_2^*\widehat{\sheaf{E}},p_1^*\sheaf{F})\tensor\sheaf{P}$ on $X\times_S\widehat{X}$.

Push forward to $\widehat{X}$ to get
\begin{align*}
p_{2*}(\HHom(p_2^*\widehat{\sheaf{E}},p_1^*\sheaf{F})\tensor\sheaf{P})
&\iso p_{2*}\HHom(p_2^*\widehat{\sheaf{E}},p_1^*\sheaf{F}\tensor\sheaf{P})\\
&\iso \HHom(\widehat{\sheaf{E}}, p_{2*}(p_1^*\sheaf{F}\tensor\sheaf{P}))\\
&\iso \HHom(\widehat{\sheaf{E}}, \widehat{\sheaf{F}})\\
\intertext{and push forward to $X$ to get, by relative duality over $p_1$,}
p_{1*}\HHom(p_2^*\widehat{\sheaf{E}},p_1^*\sheaf{F})\tensor\sheaf{P}
&\iso p_{1*}\HHom(p_2^*\sheaf{E}\tensor\sheaf{P}^\vee,p_1^*\sheaf{F})\\
&\iso \HHom(p_{1*}(p_2^*\widehat{\sheaf{E}}\tensor\sheaf{P}^\vee\tensor\omega_{p_1}[g]), \sheaf{F})\\
&\iso \HHom(\sheaf{E}, \sheaf{F})
\end{align*}
where $g$ is the relative dimension of $X/S$,
and where we used Fourier--Mukai inversion
in the last step.

Thus $\HHom(p_2^*\widehat{\sheaf{E}},p_1^*\sheaf{F})\tensor\sheaf{P}$
pushes forward to $\HHom(\sheaf{E},\sheaf{F})$ on $X$ and to
$\HHom(\widehat{\sheaf{E}},\widehat{\sheaf{F}})$ on $\widehat{X}$.
Push further down to $S$ to obtain (\ref{item:fm-hom}).

The induced map in (\ref{item:fm-ext})
now arises as a composition
\begin{equation*}
\Ext^i_X(\sheaf{E},\sheaf{F})
\iso
\Ext^i_{X\times_S\widehat{X}}(p_2^*\widehat{\sheaf{E}},p_1^*\sheaf{F}\tensor\sheaf{P})
\iso
\Ext^i_{\widehat{X}}(\widehat{\sheaf{E}},\widehat{\sheaf{F}}).
\end{equation*}
The first isomorphism takes $f\colon
\sheaf{E}\to\sheaf{F}[i]$ to
\begin{equation*}
p_2^*\widehat{\sheaf{E}} = p_2^*p_{2*}(p_1^*\sheaf{E}\tensor\sheaf{P})
\xrightarrow{\text{adj}}
p_1^*\sheaf{E}\tensor\sheaf{P}
\xrightarrow{p_1^*(f)\tensor\sheaf{P}}
(p_1^*\sheaf{F}\tensor\sheaf{P})[i]
\end{equation*}
and the second isomorphism takes $g\colon p_2^*\widehat{\sheaf{E}}\to p_1^*\sheaf{F}\tensor\sheaf{P}[i]$ to
\begin{equation*}
\widehat{\sheaf{E}} \xrightarrow{\text{adj}}
p_{2*}p_2^*\widehat{\sheaf{E}}
\xrightarrow{p_{2*}(g)}
p_{2*}(p_1^*\sheaf{F}\tensor\sheaf{P})[i] = \widehat{\sheaf{F}}[i].
\end{equation*}
The composition of these is the Fourier--Mukai transform.
\end{proof}

\begin{remark}\label{rem:symmetry}
The isomorphism
\begin{equation*}
\beta(\sheaf{E},\sheaf{F})\colon
\pi_*\HHom(\sheaf{E},\sheaf{F}) \to
\widehat{\pi}_*\HHom(\widehat{\sheaf{E}},\widehat{\sheaf{F}})
\end{equation*}
from Lemma \ref{lem:fm-hom}(\ref{item:fm-hom}) is compatible with relative duality:
the relative canonical sheaf $\omega_{X/S}$ has trivial fibres, and hence can be
written $\pi^*\eta$, where $\eta$ is the restriction of $\omega_{X/S}$ to the zero section.
On the dual side, the canonical sheaf $\omega_{\widehat{X}/S}$ is $\widehat{\pi}^*\eta$ for
the same $\eta$.
Thus relative duality gives the vertical arrows in the diagram of isomorphisms
\begin{equation*}
\begin{diagram}
\HHom(\widehat{\pi}_*\HHom(\widehat{\sheaf{E}},\widehat{\sheaf{F}}),\OO_S)
& \rTo^{\beta(\sheaf{E},\sheaf{F})^\vee} &
\HHom(\pi_*\HHom(\sheaf{E},\sheaf{F}),\OO_S) \\
\dTo & & \dTo \\
\widehat{\pi}_*\HHom(\widehat{\sheaf{F}},\widehat{\sheaf{E}})\tensor \eta[g]
& \rTo^{\beta(\sheaf{F},\sheaf{E})^{-1}\tensor\eta[g]} &
\pi_*\HHom(\sheaf{F},\sheaf{E})\tensor \eta[g]
\end{diagram}
\end{equation*}
and this diagram is commutative.
We leave out the (formal) verification of this statement.
\end{remark}

Let $\sheaf{E}$ be a perfect complex on the abelian scheme $X/S$, and let
\begin{align*}
\iota&\colon \pi_*\OO_X \to \pi_*\HHom(\sheaf{E},\sheaf{E}), &
\tr&\colon \pi_*\HHom(\sheaf{E},\sheaf{E}) \to \pi_*\OO_X
\end{align*}
denote the homomorphisms induced by the canonical ``diagonal'' map $\OO_X\to \HHom(\sheaf{E},\sheaf{E})$ and the trace map $\HHom(\sheaf{E},\sheaf{E})\to \OO_X$.
Let $\widehat{\iota}$ and $\widehat{\tr}$ denote the corresponding maps associated to $\widehat{\sheaf{E}}$ on $\widehat{X}$.

\begin{lemma}\label{lem:iota-trace}
Form the composition
\begin{equation*}
\widehat{\pi}_*\OO_{\widehat{X}} \xrightarrow{\widehat{\iota}} \widehat{\pi}_*\HHom(\widehat{\sheaf{E}},\widehat{\sheaf{E}}) \iso
\pi_*\HHom(\sheaf{E},\sheaf{E}) \xrightarrow{\tr} \pi_*\OO_X.
\end{equation*}
and take first cohomology sheaves.
The result coincides with the negative of the homomorphism
\begin{equation*}
d\phi_{\det(\sheaf{E})}\colon h^1\widehat{\pi}_*\OO_{\widehat{X}} \to h^1\pi_*\OO_X
\end{equation*}
induced by $\phi_{\det(\sheaf{E})}\colon X\to \widehat{X}$.
\end{lemma}

\begin{remark}\label{rem:tangent}
The notation $d\phi_{\det(\sheaf{E})}$ is a reminder of the following \cite[Theorem 8.4.1]{BLR90}:
if $\sigma\colon S\to X$ denotes the zero section,
then the restriction $\sigma^* T_{X/S}$
of the relative tangent bundle is canonically isomorphic to
$h^1\widehat{\pi}_*\OO_{\widehat{X}}$.
The induced homomorphism in the lemma can be viewed as the derivative
of $\phi_{\det(\sheaf{E})}$ along the zero section.
\end{remark}

\begin{proof}
The claim is local on $S$, so we may assume $S=\Spec R$ is affine.
Then the map in question is the composition of certain
($R$-module-) homomorphisms
\begin{equation}\label{eq:trace-circ-iota}
\Ext^1_{\widehat{X}}(\OO_{\widehat{X}}, \OO_{\widehat{X}})
\xrightarrow{\widehat{\iota}} \Ext^1_{\widehat{X}}(\widehat{\sheaf{E}},\widehat{\sheaf{E}})
\iso \Ext^1_X(\sheaf{E}, \sheaf{E})
\xrightarrow{\tr} \Ext^1_X(\OO_X, \OO_X).
\end{equation}
View elements of these $\Ext^1$-groups as first order infinitesimal deformations
of the argument, i.e.~families of perfect complexes over $S[\epsilon] = \Spec R[\epsilon]$,
where $\epsilon^2=0$. The key point in what follows is that the trace map sends a first order deformation
to its determinant.

Since $X$ is the dual of $\widehat{X}$,
first order deformations of $\OO_{\widehat{X}}$
correspond to morphisms $f\colon S[\epsilon] \to X$ extending
the zero section $\sigma\colon S\to X$.
In these terms, the homomorphism $d\phi_{\det(\sheaf{E})}$ in the Lemma sends
$f$ to $\phi_{\det(\sheaf{E})}\circ f$.
By definition of $\phi_{\det(\sheaf{E})}$, the corresponding deformation of $\OO_X$
is
\begin{equation}\label{eq:phi-star}
T_f^*(q^*\det(\sheaf{E})) \tensor q^*\det(\sheaf{E})^{-1} \in \Pic(X\times_S S[\epsilon])
\end{equation}
where $q$ is projection to $X$. We want to compare this with the image of $f$
through the string \eqref{eq:trace-circ-iota} of homomorphisms.

The Fourier--Mukai transform yields an isomorphism
\begin{equation*}
\Ext^1_X(\sigma_*\OO_S,\sigma_*\OO_S) \iso 
\Ext^1_{\widehat{X}}(\OO_{\widehat{X}}, \OO_{\widehat{X}}).
\end{equation*}
Let $F=(f,1)\colon S[\epsilon] \to X\times_S S[\epsilon]$.
Then the first order deformation of $\sigma_*\OO_S$ corresponding to $f$
is $F_*\OO_{S[\epsilon]}$.

Under $\widehat{\iota}$, the deformation $\widehat{F_*\OO_{S[\epsilon]}}$ is sent to the tensor product
\begin{equation*}
\widehat{F_*\OO_{S[\epsilon]}} \tensor \widehat{q^*\sheaf{E}} \in \Ext^1_X(\widehat{\sheaf{E}}, \widehat{\sheaf{E}}).
\end{equation*}
The Fourier--Mukai transform exchanges tensor product and Pontryagin product
(cf.~Mukai \cite[§3.7]{mukai87}, where the argument applies also relative to a base),
defined as $A\star B = m_*(p_1^*A \tensor p_2^*B)$.
Thus we arrive at
\begin{equation*}
(F_*\OO_{S[\epsilon]}) \star (q^*\sheaf{E}) \in \Ext^1_X(\sheaf{E},\sheaf{E}).
\end{equation*}
Using the shorthand $X[\epsilon] = X\times_S S[\epsilon]$
and with an eye at the commutative diagram
\begin{equation*}
\begin{diagram}[width=4em]
X[\epsilon] & \lTo^{p_2}_{\smash{\raisebox{-0.3ex}{$\displaystyle\widetilde{\phantom{--}}$}}} & S[\epsilon]\times_{S[\epsilon]} X[\epsilon] & \rTo^{p_1}   & S[\epsilon] \\
\dTo_{T_f} &&\dTo_{F\times 1}        &        & \dTo_F \\
X[\epsilon] &\lTo^m &X[\epsilon]\times_{S[\epsilon]} X[\epsilon] & \rTo^{p_1} & X[\epsilon]
\end{diagram}
\end{equation*}
we write out what this means:
\begin{align*}
(F_*\OO_{S[\epsilon]}) \star (q^*\sheaf{E})
&= m_*\left(p_1^*(F_*\OO_{S[\epsilon]})\tensor (p_2^*q^*\sheaf{E})\right) \\
&= m_*\left((F\times 1)_*(p_2^*q^*\sheaf{E})\right) \\
&= T_{f*}(q^*\sheaf{E}).
\end{align*}
Now $T_{f*} = T_{-f}^*$,
and the trace map applied to the first order deformation $T_{-f}^*(q^*\sheaf{E})$
yields its determinant, i.e.
\begin{equation*}
T_{-f}^*(q^*\det(\sheaf{E})) \tensor q^*\det(\sheaf{E})^{-1}.
\end{equation*}
This is, by the theorem of the square, the inverse to the invertible sheaf
\eqref{eq:phi-star}.
\end{proof}

\subsection{Splitting off traces}\label{sec:split}

To define Donaldson--Thomas invariants, we are going to construct
a perfect symmetric obstruction theory on $M(\sheaf{L}',\sheaf{L})$.
However, to obtain deformation invariance of the resulting
invariant, which is a key point, a relative version of the obstruction
theory is needed.
So we extend the setup from Section \ref{sec:chern} to the relative
situation of an abelian scheme $X/S$.

Fix integers $r$ and $\chi$ together with line bundles $\sheaf{L}$ on $X$
and $\sheaf{L}'$ on $\widehat{X}$. We say that a complex $\sheaf{E}$ on $X\times_S T$
has first Chern class $\sheaf{L}$ if $\phi_{\det(\sheaf{E})} = \phi_{\sheaf{L}}\times_S T$
as morphisms $X\times_S T \to \widehat{X}\times_S T$.

Let $M/S$ be a moduli space of perfect complexes on $X/S$ with rank $r$, first Chern class $\sheaf{L}$,
and with Fourier--Mukai transform of rank $\chi$ and first Chern class $\sheaf{L}'$.
Let $\sheaf{E}$ denote the universal family
(possibly defined only locally on $M$, as in Section \ref{sec:moduli}).
Recall that an isogeny of abelian schemes means a finite and surjective morphism
of group schemes. Such a morphism is necessarily flat \cite[Lemma 6.12]{mumford65}.

\begin{condition}\label{condition-relative}
Assume that
\begin{equation*}
\begin{pmatrix}
\chi & \phi_{\sheaf{L}'} \\
-\phi_{\sheaf{L}} & r
\end{pmatrix}
\in \End(X\times_S\widehat{X})
\end{equation*}
is an isogeny.
\end{condition}

When $S = \Spec \CC$, the endomorphism in Condition \ref{condition-relative}
agree with the one in Condition \ref{condition}.

\begin{proposition}\label{prop:condition}
If the endomorphism in Condition \ref{condition-relative} restricts to an isogeny
in some closed fibre $X_s\times \widehat{X}_s$, then it is an isogeny globally.
\end{proposition}

\begin{proof}
Let $f$ denote the endomorphism in Condition \ref{condition-relative}.
The locus of points $s\in S$ such that $f_s$ is an isogeny is open (as
the locus where $\Ker(f)\to S$ is finite)
and closed (as the locus over which $f_s$ is surjective),
and nonempty by assumption. Thus $f_s$ is an isogeny for all $s\in S$.
So $f$ is finite and surjective, i.e.~an isogeny of abelian schemes.
\end{proof}

Lemma \ref{lem:fm-hom} applies to the projection $p_2\colon X\times_S M \to M$,
which itself is an abelian scheme, with dual $p_2\colon \widehat{X}\times_S M\to M$.
Thus we identify
$p_{2*}\HHom(\sheaf{E},\sheaf{E})=p_{2*}\HHom(\widehat{\sheaf{E}},\widehat{\sheaf{E}})$,
and then there are canonical (co)diagonal and (co)trace maps
\begin{align*}
&\iota\colon p_{2*}\OO_{X\times_S M} \to p_{2*}\HHom(\sheaf{E},\sheaf{E}),
&
&\tr\colon p_{2*}\HHom(\sheaf{E},\sheaf{E}) \to p_{2*}\OO_{X\times_S M},\\
&\widehat{\iota}\colon p_{2*}\OO_{\widehat{X}\times_S M} \to p_{2*}\HHom(\sheaf{E},\sheaf{E}),
&
&\widehat{\tr}\colon p_{2*}\HHom(\sheaf{E},\sheaf{E}) \to p_{2*}\OO_{\widehat{X}\times_S M}.
\end{align*}

Form the composition $(\widehat{\tr},\tr)\circ(\widehat{\iota}+\iota)$ and
take cohomology sheaves. Since the cohomology sheaves of the complex $p_{2*}\OO_{X\times_S M}$
are $h^i(\pi_*\OO_X) \tensor_{\OO_S}\OO_M$, and similarly on the dual side, the result can be written
\begin{equation}\label{eq:composition}
\big(h^i(\widehat{\pi}_*\OO_{\widehat{X}})\oplus h^i(\pi_*\OO_X)\big) \tensor_{\OO_S} \OO_M
\to
\EExt^i_{p_2}(\sheaf{E},\sheaf{E})\\
\to
\big(h^i(\widehat{\pi}_*\OO_{\widehat{X}})\oplus h^i(\pi_*\OO_X)\big) \tensor_{\OO_S} \OO_M
\end{equation}

For $i=1$, view $h^1(\widehat{\pi}_*\OO_{\widehat{X}})\oplus h^1(\pi_*\OO_X)$ in \eqref{eq:composition}
as the relative tangent space to $X\times_S\widehat{X}$ along the zero section, as in Remark \ref{rem:tangent}.

\begin{lemma}\label{lem:composition}
For $i=1$, the composition \eqref{eq:composition} is
the endomorphism
\begin{equation}\label{eq:derivative}
\begin{pmatrix}
\chi & d\phi_{\sheaf{L}'} \\
-d\phi_{\sheaf{L}} & r
\end{pmatrix},
\end{equation}
of $h^1(\widehat{\pi}_*\OO_{\widehat{X}})\oplus h^1(\pi_*\OO_X)$, tensored with $\OO_M$.
\end{lemma}

\begin{proof}
The bottom right entry is $h^1 p_{2*}$ applied to
\begin{equation*}
\OO_{X\times_S M} \xrightarrow{\iota} \HHom(\sheaf{E},\sheaf{E})
\xrightarrow{\tr} \OO_{X\times_S M},
\end{equation*}
but this is multiplication by the rank $r$ of $\sheaf{E}$ already
before the application of the derived push forward. Similarly,
the top left entry is multiplication by the rank of
$\widehat{\sheaf{E}}$.

The bottom left entry is $-d\phi_{\det(\sheaf{E})}$ by Lemma
\ref{lem:iota-trace}, and $\phi_{\det(\sheaf{E})} = \phi_{\sheaf{L}}\times_S M$.
Exchanging the roles of $(X,\sheaf{E})$ and $(\widehat{X},\widehat{\sheaf{E}})$, we see that
the top right entry is $d\phi_{\det(\widehat{\sheaf{E}})} = d(\phi_{\sheaf{L}'}\times_S M)$.
(The sign
change comes from the $(-1)^*$ in Mukai's isomorphism \eqref{eq:fourier}.)
\end{proof}

\begin{lemma}\label{lem:split}
Assume $X/S$ has dimension $g=3$ and Condition \ref{condition-relative} holds.
Then the truncation
\begin{equation*}
\tau^{[1,2]}\big(p_{2*}\HHom(\sheaf{E},\sheaf{E})\big)
\to
\tau^{[1,2]}\big(p_{2*}\OO_{\widehat{X}\times_S M}\oplus p_{2*}\OO_{X\times_S M}\big)
\end{equation*}
of $(\widehat{\tr},\tr)$
is a split epimorphism.
\end{lemma}

\begin{proof}
It suffices to show that the truncation of
$(\widehat{\tr},\tr)\circ(\widehat{\iota}+\iota)$
to $[1,2]$ is an automorphism of
$\tau^{[1,2]}\big(p_{2*}\OO_{\widehat{X}\times_S M}\oplus p_{2*}\OO_{X\times_S M}\big)$
in the derived category.
In fact we claim: for $X/S$ of arbitrary dimension $g$,
the composition \eqref{eq:composition} is an isomorphism
for $i=1$ and $i=g-1$.

For $i=1$, the claim follows from Lemma \ref{lem:composition},
since Condition \ref{condition-relative} implies that 
\eqref{eq:derivative} is an automorphism of $h^1(\widehat{\pi}_*\OO_{\widehat{X}})\oplus h^1(\pi_*\OO_X)$.

The case $i=g-1$ follows from duality:
the complexes $\OO_{X\times_S M}$, $\OO_{\widehat{X}\times_S M}$
$\HHom(\sheaf{E},\sheaf{E})$ and $\HHom(\widehat{\sheaf{E}},\widehat{\sheaf{E}})$
are all self dual in the derived sense,
and the duals of the diagonal maps
$\iota$ and $\widehat{\iota}$,
upstairs on $X\times_S M$ and $\widehat{X}\times_S M$,
are canonically the trace maps $\tr$ and $\widehat{\tr}$.
Thus, by relative duality over $M$ and Remark \ref{rem:symmetry}, the dual of
\begin{equation*}
p_{2*}\OO_{\widehat{X}\times_S M}
\xrightarrow{\widehat{\iota}}
p_{2*}\HHom(\widehat{\sheaf{E}},\widehat{\sheaf{E}})
\overset{\beta}{\iso}
p_{2*}\HHom(\sheaf{E},\sheaf{E})
\xrightarrow{\tr}
p_{2*}\OO_{X\times_S M}
\end{equation*}
is, canonically,
\begin{equation*}
p_{2*}\OO_{\widehat{X}\times M}[g]
\xleftarrow{\widehat{\tr}[g]}
p_{2*}\HHom(\widehat{\sheaf{E}},\widehat{\sheaf{E}})[g]
\overset{\beta[g]}{\iso}
p_{2*}\HHom(\sheaf{E},\sheaf{E})[g]
\xleftarrow{\iota[g]}
p_{2*}\OO_{X\times M}[g]
\end{equation*}
tensored with a line bundle $\eta$ (as in Remark \ref{rem:symmetry}).
So, suppressing $\beta$ again, the dual of $\tr\circ\widehat{\iota}$
is $(\widehat{\tr}\circ\iota)\tensor \eta[g]$.
It follows that the dual of $(\widehat{\tr},\tr)\circ(\widehat{\iota}+\iota)$
is $(\tr,\widehat{\tr})\circ(\iota+\widehat{\iota})\tensor\eta[g]$.
Knowing that the first cohomology of
$(\widehat{\tr},\tr)\circ(\widehat{\iota}+\iota)$
is an isomorphism,
we conclude that its $g-1$'st cohomology is
an isomorphism, too.
\end{proof}

\begin{remark}
By working with moduli spaces for complexes with fixed determinant, and
the corresponding trace free obstruction theory, we only get rid of one of the trivial summands
in Lemma \ref{lem:split}, and the associated virtual fundamental
class is still zero. So in order to get a nontrivial invariant, we will fix
the codeterminant, too.
\end{remark}

\subsection{The obstruction theory}\label{sec:obstr}

Let $M(\sheaf{L}',\sheaf{L}) \subset M$ denote the
sub moduli space parametrizing complexes with fixed determinant $\sheaf{L}$
and codeterminant $\sheaf{L}'$, i.e. the fibre product
\begin{equation*}
\begin{diagram}
M(\sheaf{L}',\sheaf{L}) & \rTo & M \\
\dTo & & \dTo^\delta\\
S & \rTo & \Pic(\widehat{X}/S)\times_S\Pic(X/S)
\end{diagram}
\end{equation*}
where the lowermost horizontal arrow is defined by $(\sheaf{L}',\sheaf{L})$,
and $\delta$ is defined by $(\det(\widehat{\sheaf{E}}),\det(\sheaf{E}))$.

Recall that Behrend--Fantechi \cite{BF2008} say that a perfect obstruction theory
$\sheaf{F} \to L_M$ over $k$ is \emph{symmetric}
if it is equipped with an isomorphism $\theta\colon \sheaf{F}\to \sheaf{F}^\vee[1]$ satisfying
$\theta^\vee[1] = \theta$. We extend the definition to the relative situation
by working modulo line bundles from the base.

\begin{definition}\label{def:symmetry}
A relative perfect obstruction theory
$\sheaf{F}\to L_{M/S}$ is \emph{symmetric} if there is a line bundle $\eta$ on $S$ and an 
isomorphism $\theta\colon \sheaf{F} \to \sheaf{F}^\vee\tensor_{\OO_S} \eta[1]$ satisfying
$\theta^\vee[1]\tensor_{\OO_S} \eta = \theta$.
\end{definition}

Let $\sheaf{F}$ denote the kernel of the split epimorphism 
in Lemma \ref{lem:split}, so that
\begin{equation}\label{eq:split}
\tau^{[1,2]}p_{2*}\HHom(\sheaf{E},\sheaf{E}) \iso
\sheaf{F} \oplus \tau^{[1,2]}\left(p_{2*}\OO_{\widehat{X}\times_S M} \oplus p_{2*}\OO_{X\times_S M}\right).
\end{equation}

\begin{theorem}\label{thm:obstr}
Assume Condition \ref{condition-relative}.
The morphism \eqref{eq:obstr-morphism} induces a morphism
\begin{equation*}
\sheaf{F}^\vee[-1] \to L_{M/S},
\end{equation*}
whose restriction to $M(\sheaf{L}',\sheaf{L})$
is a relative perfect symmetric obstruction theory.
\end{theorem}

\begin{proof}
The arguments of Huybrechts--Thomas \cite[Theorem 4.1 and Section 4.4]{HT2010}
work with minor additions.
Simplicity of $\sheaf{E}$ combined with Serre duality shows that
the diagonal and trace maps induce isomorphisms
\begin{gather*}
\OO_M = h^0p_{2*}\OO_{X\times_S M} \isoto h^0 p_{2*}\HHom(\sheaf{E},\sheaf{E}),\\
h^3p_{2*}\HHom(\sheaf{E},\sheaf{E}) \isoto h^3 p_{2*}\OO_{X\times_S M}
\iso \eta^\vee\tensor_{\OO_S}\OO_M
\end{gather*}
(the last isomorphism by relative duality for $\pi$; as before $\eta$ is the line bundle
for which $\omega_{X/S} = \pi^*\eta$)
and thus there are distinguished triangles
\begin{gather}\label{eq:triangles}
\OO_M \to p_{2*}\HHom(\sheaf{E},\sheaf{E}) \to \tau^{\ge 1}p_{2*}\HHom(\sheaf{E},\sheaf{E}),
\notag\\
\tau^{[1,2]}p_{2*}\HHom(\sheaf{E},\sheaf{E})
 \to \tau^{\ge 1}p_{2*}\HHom(\sheaf{E},\sheaf{E}) \to \eta^\vee\tensor_{\OO_S}\OO_M[-3].
\end{gather}

By these triangles (and the fact that $L_{M/S}$ is,
by definition, concentrated in degrees $[-1,0]$),
the morphism \eqref{eq:obstr-morphism} induces
\begin{equation}\label{eq:trunc-obstr}
(\tau^{[1,2]}p_{2*}\HHom(\sheaf{E},\sheaf{E}))^\vee[-1] \to L_{M/S}
\end{equation}
and hence, by the direct sum decomposition \eqref{eq:split},
\begin{equation}\label{eq:my-obstr}
\sheaf{F}^\vee[-1] \to L_{M/S}.
\end{equation}

For every deformation situation over $S$,
\begin{equation*}
f\colon T \to M,\quad T \subset \overline{T}
\end{equation*}
where $T\subset \overline{T}$ is an affine square zero thickening with
ideal $\sheaf{I}\subset \OO_{\overline{T}}$, the map \eqref{eq:trunc-obstr}
gives rise to a class
\begin{equation*}
\omega\in \Ext^2(f^*(\tau^{[1,2]}p_{2*}\HHom(\sheaf{E},\sheaf{E}))^\vee, \sheaf{I})
\iso H^2(\HHom(\sheaf{E}_T,\sheaf{E}_T)\tensor_{\OO_T} \sheaf{I})
\end{equation*}
which is in fact an obstruction class \cite[Main theorem]{HT2010}, i.e., it vanishes if and only if
$f$ extends to $\overline{T}$.
Now $\omega$ decomposes as
\begin{equation*}
\omega = (\omega_0,\omega_1,\omega_2)
\in
H^2(f^*\sheaf{F}\tensor\sheaf{I}) \oplus H^2(\sheaf{I}) \oplus H^2(\sheaf{I}).
\end{equation*}
Here,
\begin{equation*}
\omega_0 \in \Ext^2(f^*(\sheaf{F}^\vee), \sheaf{I})
\iso
H^2(f^*\sheaf{F}\tensor\sheaf{I})
\end{equation*}
is the class induced by \eqref{eq:my-obstr}, and
$\omega_1$ and $\omega_2$ are the obstructions to deforming
$\det(\sheaf{E}_T)$ and $\det(\widehat{\sheaf{E}}_T)$, as
the trace map sends the obstruction class of a perfect complex
to the obstruction class of its determinant.
Since line bundles deform freely in characteristic zero, $\omega_1$ and $\omega_2$ vanish,
and so $\omega_0$ vanishes if and only $\omega$ does.
Thus $\omega_0$ is an obstruction class.
If in addition $\det(\sheaf{E}_T)$ is constant (in the sense of being a pullback from the base),
then the set of extensions $\overline{T} \to M$ of $f$ with
constant determinant is a torsor under the natural action of
the trace free part of
\begin{equation*}
\Ext^1(f^*(\tau^{[1,2]}p_{2*}\HHom(\sheaf{E},\sheaf{E}))^\vee, \sheaf{I})
\iso H^1(\HHom(\sheaf{E}_T,\sheaf{E}_T)\tensor_{\OO_T} \sheaf{I}).
\end{equation*}
Thus, if $f$ factors through $M(\sheaf{L}',\sheaf{L})$, so that
both $\det(\sheaf{E}_T)$ and $\det(\widehat{\sheaf{E}}_T)$ are constant,
then extensions $\overline{T}\to M(\sheaf{L}',\sheaf{L})$ of $f$ form a torsor under
the trace- and cotrace free part of $H^1(\HHom(\sheaf{E}_T,\sheaf{E}_T)\tensor_{\OO_T} \sheaf{I})$,
which is $H^1(f^*\sheaf{F}\tensor\sheaf{I})$.
Thus the restriction of \eqref{eq:my-obstr} to $M(\sheaf{L}',\sheaf{L})$ is a relative obstruction theory.

Whenever two terms in a distinguished triangle is perfect,
so is the third, and thus $\sheaf{F}$ is perfect.
For every point $i\colon \Spec k \to M$,
the vector space $h^p(i^*\sheaf{F})$ is the trace- and cotrace free part
of $\Ext^p(i^*\sheaf{E},i^*\sheaf{E})$, and in particular vanishes for $p\not\in [1,2]$.
Thus, by an application of Nakayama's Lemma,
the perfect complex $\sheaf{F}$ has amplitude $\subseteq [1,2]$.
By the two distinguished triangles \eqref{eq:triangles}, it follows that relative duality
\begin{equation*}
(p_{2*}\HHom(\sheaf{E},\sheaf{E}))^\vee \iso
p_{2*}\HHom(\sheaf{E},\sheaf{E})\tensor_{\OO_S} \eta[3]
\end{equation*}
induces
\begin{equation*}
(\tau^{[1,2]}p_{2*}\HHom(\sheaf{E},\sheaf{E}))^\vee \iso
(\tau^{[1,2]}p_{2*}\HHom(\sheaf{E},\sheaf{E}))\tensor_{\OO_S}\eta[3].
\end{equation*}
Since the direct summand $\sheaf{F}$ is both the kernel of the trace\slash cotrace map
and the cokernel of the diagonal\slash codiagonal map, and these two maps are dual,
there is an induced isomorphism $\sheaf{F}^\vee \iso \sheaf{F}\tensor_{\OO_S}\eta[3]$.
Thus the obstruction theory \eqref{eq:my-obstr} (restricted to $M(\sheaf{L}',\sheaf{L})$)
is perfect and symmetric.
\end{proof}

\section{Perfect complexes modulo twist and translate}\label{sec:K}

Let $X$ be an abelian threefold over $S = \Spec k$.

With $X\times\widehat{X}$ acting on a moduli space $M$ of perfect complexes
as before, by translation and twist, define the stack quotient
\begin{equation*}
K = [M/X\times \widehat{X}].
\end{equation*}
In this section we rephrase the results from the previous section,
and as outcome we obtain a numerical invariant of $K$.
This is our Donaldson--Thomas invariant for abelian threefolds.

We pause to explain our point of view: suppose $Y$ is a second abelian variety,
derived equivalent to $X$, i.e.\ there is an equivalence $F\colon D(X) \isoto D(Y)$
of triangulated categories. By Orlov \cite{orlov2002}, there is
an object $\sheaf{C}\in D(X\times Y)$ on the product, such that $F$ is
the Fourier--Mukai functor induced by $\sheaf{C}$:
\begin{equation*}
F(-) = p_{2*}(p_1^*(-)\tensor\sheaf{C})
\end{equation*}
Consequently the equivalence $F$ extends in a natural way to families,
i.e.\ there are equivalences $F_T\colon D(X\times T)\to D(Y\times T)$,
functorial in $T$. In particular, if $\sheaf{E}$ is the universal family
on $X\times M$, then we may replace $\sheaf{E}$ by $F_M(\sheaf{E})$
and view $M$ as a moduli space for complexes on $Y$. We wish to treat
$(X,\sheaf{E})$ and $(Y,F(\sheaf{E}))$ on an equal footing, in particular
our Donaldson--Thomas invariants should be invariant under derived equivalence.

The equivalence $F$ induces an isomorphism
\begin{equation*}
f\colon X\times\widehat{X} \isoto Y\times\widehat{Y}
\end{equation*}
such that \cite[Corollary 2.13]{orlov2002}
\begin{equation*}
T_y^*F(-)\tensor \zeta \iso F(T_x^*(-)\tensor\xi)
\end{equation*}
whenever $(y,\zeta) = f(x,\xi)$. In particular the actions of $X\times\widehat{X}$
and $Y\times\widehat{Y}$ on $M$ are compatible, and there are induced isomorphisms
\begin{equation*}
[M/X\times\widehat{X}] \iso [M/Y\times\widehat{Y}],
\end{equation*}
so the construction of $K$ is invariant under derived equivalence.
(In contrast, the space $M(\sheaf{L}',\sheaf{L})$ is not.)
This setup seems to be natural also for not necessarily abelian varieties:
Rosay \cite{rosay2009} has shown that, for any smooth projective variety $X$,
the identity component of the algebraic group
$\Aut(D(X))$ of autoequivalences is in fact $\Aut^0(X)\times\Pic^0(X)$, and by an
unpublished result of Rouquier \cite[Proposition 9.45]{huybrechts2006}, any
equivalence $F\colon D(X)\isoto D(Y)$ induces an isomorphism
$\Aut^0(X)\times\Pic^0(X)\isoto
\Aut^0(Y)\times\Pic^0(Y)$.

\begin{proposition}\label{prop:DM}
Assume Condition \ref{condition-relative}, and let $G$ be the kernel of the isogeny appearing there.
Then there is an isomorphism
\begin{equation*}
[M(\sheaf{L}',\sheaf{L})/G] \isoto K.
\end{equation*}
In particular $K$ is Deligne-Mumford.
\end{proposition}

\begin{proof}
As we saw in the proof of Proposition \ref{prop:isotrivial}, the determinant\slash codeterminant
morphism \eqref{eq:delta} is $X\times\widehat{X}$-equivariant, when the action on the
base is defined via the matrix in Condition \ref{condition-relative}. The statement follows.
\end{proof}

Although it is easy to verify that the obstruction theory in Theorem \ref{thm:obstr}
is $G$-equivariant as an element in the derived category of $M(\sheaf{L}',\sheaf{L})$,
this does not suffice to conclude that there is an induced obstruction theory
on the quotient $K$. Bypassing this difficulty, we define the Donaldson--Thomas
invariant directly, as what it is if the obstruction theory does descend:
let $[M(\sheaf{L}',\sheaf{L})]^{\mathrm{vir}}$ be the virtual fundamental class
associated to the obstruction theory of Theorem \ref{thm:obstr}.

\begin{definition}
Assume Condition \ref{condition-relative} and that $M(\sheaf{L}',\sheaf{L})$ is proper over $k$.
The \emph{Donaldson--Thomas invariant} of the Deligne-Mumford stack $K$ is
\begin{equation*}
\DT(K) = \frac{1}{|G|} \deg [M(\sheaf{L}',\sheaf{L})]^{\mathrm{vir}}.
\end{equation*}
\end{definition}

\begin{corollary}[of Theorem \ref{thm:obstr}]\label{cor:deform}
Let $X/S$ be an abelian scheme of relative dimension $3$, and let $M/S$ be
a proper moduli space for perfect complexes, satisfying Condition \ref{condition-relative}.
Then the Donaldson--Thomas invariant $\DT(K_s)$
associated to $K_s = [M_s/X_s\times\widehat{X}_s]$ is independent of $s\in S$.
\end{corollary}

\begin{proof}
Since the endomorphism in Condition \ref{condition} is an isogeny, its kernel $G$
is flat and finite over $S$, so its order $|G_s|$ is constant in $s$.

The existence of a relative obstruction theory implies
(see \cite[Proposition 7.2]{BF97} and \cite[Corollary 4.3]{HT2010})
that the degree
of the virtual fundamental class of $M_s(\sheaf{L}'_s,\sheaf{L}_s)$ is constant in $s$.
Hence the Donaldson--Thomas invariant
\begin{equation*}
\DT(K_s) = \frac{1}{|G_s|} \deg [M_s(\sheaf{L}'_s,\sheaf{L}_s)]^{\mathrm{vir}}.
\end{equation*}
is independent of $s$.
\end{proof}

Recall Behrend's theorem \cite{behrend2009} on weighted Euler characteristics: whenever
a proper $k$-scheme (or Deligne-Mumford stack) $M$ admits a perfect symmetric
obstruction theory, the degree of the associated virtual fundamental class agrees with
the weighted Euler characteristic
\begin{equation*}
\tilde{\chi}(M) = \sum_{n\in \ZZ} n \chi(\nu^{-1}(n))
\end{equation*}
where $\chi$ is the usual topological Euler characteristic,
and $\nu\colon M\to \ZZ$ is Behrend's invariant of singularities \cite[Definition 1.4]{behrend2009}.

\begin{corollary}[of Behrend's theorem]
Assume Condition \ref{condition}.
The Donaldson--Thomas invariant $\DT(K)$ agrees with Behrend's weighted Euler
characteristic $\tilde{\chi}(K)$, and hence is an intrinsic invariant of
the Deligne-Mumford stack $K$.
\end{corollary}

\begin{proof}
The quotient map $M(\sheaf{L}',\sheaf{L}) \to K$ is \'etale of degree $|G|$,
so
\begin{equation*}
\DT(K) = \frac{1}{|G|} \tilde{\chi}(M(\sheaf{L}',\sheaf{L})) = \tilde{\chi}(K).
\end{equation*}
\end{proof}

We return to Condition \ref{condition}, with the viewpoint that we are free
to replace $(X,\sheaf{E})$ with a derived equivalent pair. A coherent
sheaf $\sheaf{E}$, on an abelian variety of dimension $g$, is \emph{semi-homogeneous},
if the locus
\begin{equation}\label{eq:phi}
\Phi(\sheaf{E}) = \left\{ (x,\xi) \in X\times\widehat{X} \ \vline\ T_x^*\sheaf{E} \iso \sheaf{E}\tensor\sheaf{P}_\xi \right\}
\end{equation}
has dimension $g$ (the maximal possible) \cite{mukai78, mukai94}.
The proof of Proposition \ref{prop:isotrivial} shows that Condition \ref{condition}
implies that $\Phi(\sheaf{E})$ is finite, so semi-homogeneous sheaves fail the condition.
Conversely:

\begin{proposition}\label{prop:semihom}
Let $X$ be an abelian threefold over $\CC$ with Picard number $1$.
If $\sheaf{E}$ is a coherent sheaf on $X$,
such that $F(\sheaf{E})$ fails Condition \ref{condition} for all
autoequivalences $F$ of $D(X)$, then $\sheaf{E}$ is semi-homogeneous.
\end{proposition}

\begin{proof}
Semi-homogeneity is invariant under derived equivalence, so we are free
to replace $\sheaf{E}$ with any transform $F(\sheaf{E})$.

The twist of any sheaf with a sufficiently positive divisor is IT$_0$,
and the Fourier--Mukai transform of an IT$_0$-sheaf is a vector bundle.
Thus we may assume $\sheaf{E}$ is a vector bundle, using the derived
equivalence $F$ defined as the composition of a very positive twist,
the Fourier--Mukai transform, another very positive twist on $\widehat{X}$,
and the Fourier--Mukai transform back to $X$.

Semi-homogeneous vector bundles are numerically characterized \cite{yang89}
by having Chern character $\ch = r\exp(c_1/r)$. In dimension $3$, this means
that $r\gamma = c_1^2/2$ and $r\chi = \gamma c_1/3$. Since the Picard number is $1$,
we have $\alpha(\sigma,\tau) = \sigma\tau/3$ for any two complementary cycles $\sigma$
and $\tau$, by Remark \ref{rem:alpha}.
Thus
\begin{equation*}
r\chi - \alpha(\gamma,c_1) = 0
\end{equation*}
implies $r\chi = \gamma c_1/3$. Let $H$ be an ample divisor, denote its cohomology
class by $h$, and use primed symbols to denote the components of the
Chern character $\ch(\sheaf{E}) \exp(h)$ of $\sheaf{E}(H)$.
A small computation, using basic properties of $\alpha$ \cite[Proposition 1]{matsusaka59}, shows that
\begin{equation*}
\big(r\chi - \alpha(\gamma,c_1)\big) -
\big(r'\chi' - \alpha(\gamma',c_1')\big) 
=
(r\gamma - c_1^2/2)h - \alpha(r\gamma - c_1^2/2,h)
\end{equation*}
(this holds without any condition on the Picard number). By assumption, the left
hand side is zero, and the vanishing of the right hand side implies that $r\gamma = c_1^2/2$.
\end{proof}

In relation to the proposition, one may ask whether there exist honest
simple semi-homogeneous complexes.  Here is a weaker, but easy
observation:

\begin{remark}
Let $X$ be an abelian variety of dimension $g$, with a
perfect complex $\sheaf{E}$ having no \emph{twisted} negative Ext-groups:
\begin{equation*}
\Ext^{-p}(\sheaf{E},\sheaf{E}\tensor\sheaf{P}_\xi) = 0,\quad
\text{$\forall$ $p>0$ and $\xi\in\widehat{X}$}
\end{equation*}
Suppose $\sheaf{E}$ is semi-homogeneous in the sense that the locus $\Phi(\sheaf{E})$
in Equation \eqref{eq:phi} has dimension $g$.
We claim that $\sheaf{E}$ is isomorphic to a single semi-homogeneous sheaf,
up to shift.

In fact, since $\Phi(\sheaf{E}) \subseteq \Phi(h^p(\sheaf{E}))$,
each $h^p(\sheaf{E})$ is semi-homogeneous as well.
By Mukai's classification
of semi-homogeneous vector bundles \cite[Section 6]{mukai78} (and a little care to cover non locally free semi-homo\-ge\-neous
sheaves), each $h^p(\sheaf{E})$ admits a Jordan-H\"older filtration
with simple semi-homogeneous factors, and all these factors agree up to twist by
$\widehat{X}$. Thus if $a$ is minimal and $b$ is maximal such that
$h^p(\sheaf{E}) \ne 0$ for $p=a,b$, there are a simple semi-homogeneous
subsheaf $\sheaf{F}\subseteq h^a(\sheaf{E})$, an element $\xi\in\widehat{X}$,
and a surjection $h^b(\sheaf{E})\onto\sheaf{F}\tensor\sheaf{P}_\xi$.
The composition
\begin{equation*}
\sheaf{E}[-b] \to h^b(\sheaf{E}) \onto \sheaf{F}\tensor \sheaf{P}_\xi \into h^a(\sheaf{E}\tensor\sheaf{P}_\xi) \to \sheaf{E}\tensor\sheaf{P}_\xi[-a]
\end{equation*}
is an element in $\Ext^{a-b}(\sheaf{E},\sheaf{E}\tensor\sheaf{P}_\xi)$,
and it is nonzero since it induces a nonzero map on $h^0$.
By hypothesis we cannot have $a-b<0$, so $a=b$.
\end{remark}

\section{Examples}\label{sec:examples}

\subsection{Picard bundles}

Let $i\colon C \into X$ be a smooth projective curve embedded in its Jacobian $X=J(C)$
via the Abel-Jacobi map, and let $\sheaf{L}$ be a line bundle on $C$.
If $\sheaf{L}$ is sufficiently positive or negative,
the Fourier--Mukai transform $\widehat{i_*\sheaf{L}}$ is a (possibly shifted) vector bundle:
such bundles are called Picard bundles.
One of Mukai's original applications for his transform \cite{mukai81} was to equate deformations
of Picard bundles $\widehat{i_*\sheaf{L}}$ with deformations of $i_*\sheaf{L}$:
let $M$ be the connected component containing $i_*\sheaf{L}$, of the Simpson moduli space
for stable sheaves with respect to an arbitrary polarization.
The action
\begin{equation}\label{eq:picard-iso}
X\times\widehat{X} \to M,\quad (x,\xi)\mapsto T_{-x}^*(i_*\sheaf{L})\tensor\sheaf{P}_\xi
\end{equation}
is an \emph{isomorphism} whenever $C$ is nonhyperelliptic or has genus two \cite[Theorem 4.8]{mukai81}.
Thus
$K=[M/X\times\widehat{X}]$ is a reduced point and
\begin{align*}
\tilde\chi(K) &= \nu(\text{point}) = 1
& &\text{($C$ nonhyperelliptic or $g(C)=2$)}
\intertext{%
For hyperelliptic $C$
of genus $g>3$, the map \eqref{eq:picard-iso} is an isomorphism onto $M_{\text{red}}$,
but $M$ is everywhere nonreduced \cite[Example 1.15]{mukai87} (this
comes from obstructed first order deformations of $C$ in $X$). Thus
$K$ is a point with a nonreduced scheme structure and}
\tilde\chi(K) &= \nu(\text{nonreduced point}) > 1
& &\text{($C$ hyperelliptic and $g(C)>2$).}
\end{align*}
There is, of course, nothing ``virtual'' about these counts,
which make sense for any dimension $g$.

\subsection{Hilbert schemes of points}

Let $M$ be the moduli space for coherent sheaves with rank $1$,
$c_1=0$, $\gamma=0$ and $\chi=-n$, on an abelian threefold $X$.
Any such sheaf is a twisted ideal $\sheaf{I}_Z\tensor\sheaf{P}_\xi$,
where $Z\subset X$ is a finite subscheme of length $n$ and $\xi\in \widehat{X}$.
Thus
\begin{equation*}
M = \widehat{X} \times \Hilb^n(X).
\end{equation*}
Projection to $\widehat{X}$ agrees with the determinant map on $M$,
and projection to the Hilbert scheme followed by the summation map
\begin{equation}\label{eq:sum}
\Hilb^n(X) \to X
\end{equation}
agrees, up to sign, with the codeterminant map.
Thus $M(\OO_{\widehat{X}},\OO_X)$ is the fibre over
$0$ for \eqref{eq:sum},
and will henceforth be denoted $K^n(X)$
(this would have been Beauville's generalized Kummer variety if $X$ were a surface).
Condition \ref{condition} holds, and the kernel of the isogeny there is
the group of $n$-torsion points
\begin{equation*}
X_n = X_n\times 0 \subset X\times \widehat{X}.
\end{equation*}
Thus $K = [K^n(X)/X_n]$.

Behrend--Fantechi \cite{BF2008} show that for Hilbert schemes of points
on threefolds, the weighted Euler characteristic
and the topological Euler characteristic agree up to sign.
Their method is easily adapted to $K^n(X)$ and gives
\begin{equation*}
\tilde{\chi}(K^n(X)) = (-1)^{n+1} \chi(K^n(X)).
\end{equation*}

We also mention that Cheah's formula \cite{cheah96} for the Euler characteristic of any Hilbert scheme
of points leads us to conjecture that, for $X$ an abelian variety of dimension $g$,
\begin{equation*}
\exp\Big( \sum_{n=1}^\infty \frac{\chi(K^n(X))}{n^{2g}}t^n \Big)
= 
\sum_{n=0}^\infty P_{g-1}(n) t^n
\end{equation*}
where $P_d(n)$ is the number of $d$-dimensional partitions of $n$. (For $g=1$ and $g=2$,
the formula does hold.)
Using MacMahon's generating function for plane partitions, the conjectured
formula for $g=3$ can be written
\begin{equation*}
\chi(K^n(X)) = n^5 \sum_{d|n} d^2
\end{equation*}
giving, still conjecturally,
\begin{equation*}
\DT(K) = \frac{(-1)^{n+1}}{|X_n|} \chi(K^n(X)) = \frac{(-1)^{n+1}}{n} \sum_{d|n} d^2.
\end{equation*}

\bibliographystyle{amsplain}
\bibliography{all}

\end{document}